\documentclass[10pt]{amsart}
\usepackage{amssymb}
\usepackage{amsmath}
\usepackage{amsthm}
\usepackage{graphicx}
\usepackage{color}

\title[A normal form]{A normal form around a Lagrangian submanifold of radial points}
\author{Nick Haber}
\date{\today}
\thanks{The author was partially supported by the Department of
  Defense (DoD) through the National Defense Science \& Engineering
  Graduate Fellowship (NDSEG) Program and a National Science
  Foundation Graduate Research Fellowship under Grant No. DGE-0645962}

\numberwithin{equation}{section}
\newtheorem{thm}{Theorem}[section]
\newtheorem{lem}[thm]{Lemma}
\newtheorem{prop}[thm]{Proposition}

\theoremstyle{remark}
\newtheorem*{rmk}{Remark}
\theoremstyle{definition}
\theoremstyle{definition}
\newtheorem{define}[thm]{Definition}

\def\WF{\mathrm{WF}}

\def\ccot{\overline{T}^*X}
\def\ccotr{\overline{T}^*\R^n}

\newcommand\dpoiss[1]{ \{ \{ #1 \} \} }

\newcommand{\R}{\mathbb R}

\newcommand{\C}{\mathbb C}
\newcommand{\Z}{\mathbb Z}

\newcommand{\II}{\mathcal I}

\newcommand{\RR}{\mathcal R}

\def\Id{\operatorname{Id}}

\begin{document}

\maketitle

\begin{abstract}
In this work we produce microlocal normal forms for
pseudodifferential operators which have a Lagrangian submanifold of
radial points. This answers natural questions about such operators and
their associated classical dynamics. In a sequel, we will give a
microlocal parametrix construction, as well as a construction of a
microlocal Poisson operator, for such pseudodifferential operators.
\end{abstract}

\section{Introduction} \label{sect:introduction}

This paper gives a microlocal normal form for a homogeneous
pseudodifferential operator $P$ with real valued homogeneous principal
symbol $\sigma_m(P)$, when the Hamilton vector field corresponding to $\sigma_m(P)$ is
radial on a Lagrangian submanifold of radial points. We show that any such operator is microlocally equivalent (see Definition~\ref{def:microlocallyequivalent}) to 
$$z D_z + p_0(y)$$
on $\R_z \times \R^{n-1}_y$, with $p_0$ a smooth function, around $z =
0, y = 0, \eta = 0, \zeta = 1$. Here $\zeta$ is dual to $z$, $\eta$ is
dual to $y$, and $D_z = \frac{1}{i} \partial_z$
.
Radial points have received considerable attention in various contexts. In the setting of the
homogeneous pseudodifferential calculus, Guillemin and Schaeffer \cite{guilleminschaeffer}
gave a microlocal normal form around an isolated ray of radial points under
certain generic assumptions and studied propagation of singularities
under further assumptions. Vasy \cite{kerr-desitter} studied the
propagation of singularities around a conic Lagrangian submanifold of radial
points in the homogeneous and high-energy contexts, and Vasy and the author \cite{lag_rad2}
microlocalized these results in the homogeneous setting.
Bony, Fujii\'{e}, Ramond, and Zerzeri
\cite{bony} have results for an analogous situation in the semiclassical setting. Radial points
play an important role in analysis on scattering manifolds, first
introduced by Melrose \cite{melroseasymp}. Relevant analysis
around radial points in the scattering setting includes the works of
Melrose and Zworski \cite{melrosezworski}, Herbst and Skibsted
\cite{herbstskibstedscattering, herbstskibstedabsence}, Hassell, Melrose, and Vasy
\cite{hmv1, hmv2}, and Hassell and Vasy
\cite{hassellvasy_spectral_projections}.

{\em Acknowledgements}
I would like to thank Andr\'{a}s Vasy for countless helpful
conversations, as well as Richard Melrose for his helpful comments.

\section{Background and statement of result}

We first define this notion of radial points. Let $X$ be a smooth
manifold, and let $o$ be the zero section of $T^*X$. Denote by $\mu :
T^*X \backslash o  \times \R_{> 0} \rightarrow T^*X \backslash o$ the natural dilation
of the fibers: given $v \in T^*_x X$, $v \neq 0$, 

\begin{equation} \label{R+action}
\mu((x,v), t) = (x,tv).
\end{equation}

 We let $R$ be the vector field on $T^*X \backslash o$ which
generates this action, that is, $R$ is defined by
$$f(\cdot) \mapsto \frac{d}{dt}|_{t = 1} f(\mu(\cdot, t))$$
for $f \in C^\infty(T^*X \backslash o)$. We call $R$ the {\em radial
  vector field}. In local canonical coordinates $(x, \xi)$, $R = \xi
\cdot \partial_\xi$.

We call a set $\Gamma \subseteq T^*X \backslash o$ {\em conic} if this
$\R_{>0}$-action acts on it. We call a function $f \in C^\infty(T^*X
\backslash 0)$ homogeneous of degree $m$ if $f(x, tv) = t^m f(x, v)$ for $t > 0$.

{\em Throughout the paper, we let $\Psi^m(X) = \Psi^m_{\mathrm{cl}}(X)$
denote the space of classical pseudodifferential operators on $X$ of order $m$, that is,
pseudodifferential operators with polyhomogeneous 1-step expansions.} Given $P \in \Psi^m(X)$ with
homogeneous real-valued principal symbol $\sigma_m(P)$, we let $H_{\sigma_m(P)}$ denote the
corresponding Hamilton vector field.

\begin{define}
We say that $P$ and $H_{\sigma_m(P)}$ are {\em radial} at $q \in T^*X \backslash
o$ if $H_{\sigma_m(P)}$ is a scalar multiple of the radial vector field at $q$,
and we call $q$ a {\em radial point} (of $P$).
\end{define}

\subsection{Previous results on normal forms around radial points} \label{sect:previous}

Here we highlight two previous results on microlocal normal forms
around radial points. Both of these correspond to isolated radial
points, one in the homogeneous setting, and the other in the
scattering setting. In the homogeneous setting, this means that there
is an isolated ray of radial points, as the set of radial points is
conic. In the scattering setting, this means that there is an isolated
zero of the rescaled scattering Hamilton vector field on the boundary
`at infinity' (these two settings are are equivalent, as can be seen by a local
Fourier transform - see \cite[Section 3.1]{hmv2}).

We briefly say what we mean by microlocal normal form in this
context. Given $P \in \Psi^m(X)$, we let $\WF'(P) \subset T^*X
\backslash o$ be the microsupport of $P$, that is, $q \notin \WF'(P)$ if
there exists $A \in \Psi^0(X)$ elliptic at $q$ such that $AP \in
\Psi^{-\infty}(X)$. We refer the reader to H\"{o}rmander \cite{hormander_fio1} for
the definition and basic properties of Fourier integral operators. The notion of `homogeneous symplectomorphism' is standard. For us it is enough to say that given smooth manifolds $X$ and $Y$, there are the $\R_{>0}$ dilations $\mu_X$ and $\mu_Y$ on $T^*X \backslash o$ and $T^*Y \backslash o$, respectively, as in \eqref{R+action}. We say that a symplectomorphism 
$$f: T^*X \backslash o \rightarrow T^*Y \backslash o$$
is homogeneous if $f \circ \mu_X = \mu_Y \circ f$.

\begin{define} \label{def:microlocallyequivalent}
Given smooth manifolds $X$ and $Y$ of the same dimension and points
$q_X \in T^*X \backslash o, q_Y \in T^*Y \backslash o$, we say that
(the microlocal germ of) $P_X \in \Psi^{m_X}(X)$ at $q_X$ is {\em microlocally
  equivalent} to (the microlocal germ of) $P_Y \in \Psi^{m_Y}(Y)$ at $q_Y$ if there are
\begin{itemize}
 \item a Fourier integral operator $F$ of order $0$ quantizing a local
   homogeneous symplectomorphism sending a neighborhood of $q_Y$ to a
   neighborhood of $q_X$ (and sending $q_Y$ to $q_X$) which is
   elliptic at the point corresponding to $(q_X, q_Y)$ in the
   canonical relation, and
\item $E \in \Psi^{m_X - m_Y}(X)$, elliptic at $q_X$,
\end{itemize}
so that
$$q \notin \WF'(P_X - E F P_Y F^{-1}).$$
Here $F^{-1}$ is a microlocal parametrix of $F$ at the point
corresponding to $(q_X, q_Y)$ in the canonical relation. 
\end{define}

We follow the standard convention (see for instance
\cite{guilleminschaeffer, hmv2}) in calling explicit operators which
are microlocally equivalent to a class of operators of interest a
(microlocal) {\em normal form}. For instance, Duistermaat and H\"{o}rmander
\cite{duistermaathormander} showed that if $P \in \Psi^m(X)$ is not radial at $q$
then (the microlocal germ of) $P$ at $q$ is
microlocally equivalent to
$$D_{x_n} = \frac{1}{i} \partial_{x_n}$$
on $\R^n$ at $x = 0, \xi_n = 1, \xi_i = 0, i \neq n$, where $\dim X =
n$.

The first result we highlight is due to Guillemin and Schaeffer
\cite{guilleminschaeffer}, who launched the study of radial
points. They study the class of classical pseudodifferential operators with real principal
symbol and an isolated ray of radial points, under a generic assumption on the Hamilton
vector field (really an assumption on the linear part of it, which may
be thought of as a nonresonant assumption - see Sternberg
\cite{sternberg_local_1, sternberg_local_2, sternberg_local_3} for the
original work and Nelson \cite[Section 3]{nelson_flows} for further exposition
useful for our purposes). They show
that such an operator is, at any point on the ray of radial points,
microlocally equivalent to
$$\sum_{i, j < n} \biggl( a_{ij} D_{x_i} D_{x_j} + b_{i j} x_j D_{x_j}
D_{x_n} + c_{ij} x_i x_j D^2_{x_n} \biggr) + u x_n D^2_{x_n} + \lambda
D_{x_n}$$
on $\R^n$ at $x = 0, \xi_n = 1, \xi_i = 0, i \neq n$, for some
$a_{ij}, b_{ij}, c_{ij}, u \in \R, \lambda \in \C$ \cite[Theorem
4.1]{guilleminschaeffer}. They go on to prove propagation of
singularities results under a further limited class of such operators
- limited because, even under the nonresonant assumption, there is a
wide variety of associated classical dynamics which complicates the
picture.

We should note here that Herbst and Skibsted \cite{herbstskibstedscattering,
  herbstskibstedabsence} consider a similar situation, but in the
scattering setting.

The second result we highlight is due to Hassell, Melrose, and Vasy
\cite{hmv2}, in the scattering setting. In this
work, they remove the generic nonresonant assumptions, and in doing
so pick up resonant terms in their normal forms. These
further complicate, for instance, the microlocal form of solutions.

\subsection{Lagrangian submanifolds of radial points}

As noted earlier, in this paper we focus on the situation in which the
operator has a conic Lagrangian submanifold of radial points, in
contrast to an isolated ray of radial points. This comes up naturally in
several contexts. For instance, these arise in the scattering setting
for $\Delta - \lambda$,
$\lambda \in \R_{> 0}$ on a scattering manifold, as first seen by Melrose
\cite{melroseasymp}. They also appear in the work of Vasy
\cite{kerr-desitter, vasy_asymptotically_hyperbolic,
  vasy_analytic_continuation} in a variety of ways. An essential theme of this paper is that
analysis around a Lagrangian submanifold of radial points is much
simpler than analysis around an isolated ray of radial points. Indeed,
in the scattering setting, this is the free particle case (and the
case with a sufficiently decaying potential - see \cite{melroseasymp})
, whereas
isolated radial points arise when a symbolic potential of order $0$ is
added \cite{hmv1, hmv2}.

So that we can refer to all assumptions succinctly later, we define
the class $\RR^m(X, \Lambda)$, given a smooth manifold $X$, which we fix
to have dimension $n$, and a conic Lagrangian submanifold $\Lambda
\hookrightarrow T^*X \backslash o$. We say that $P \in \RR^m(X,
\Lambda)$ if
\begin{itemize}
  \item $P \in \Psi^m(X)$ (classical, i.e. polyhomogeneous 1-step),
   \item $P$ has homogeneous real-valued principal symbol $\sigma_m(P)$,
     \item $\Lambda \hookrightarrow \Sigma(P)$, the characteristic
       variety, 
     \item $H_{\sigma_m(P)}$ is radial on $\Lambda$, and
       \item $d(\sigma_m(P))$ does not vanish on $\Lambda$.
\end{itemize}

\subsection{Statement of result}

In \cite{lag_rad2}, Vasy and the author analyze the propagation of
singularities for $\RR^m(X, \Lambda)$, generalizing results of Melrose
\cite[Section 8]{melroseasymp} and Vasy \cite[Section
2.4]{kerr-desitter}. These propagation of singularity statements
depend on the value of the imaginary part of the subprincipal symbol of $P \in
\RR^m(X, \Lambda)$ at the point of interest. As stated more formally
and generally in
\cite[Theorem 1.4]{lag_rad2}, below a certain threshold determined by
$m$ and the value of the imaginary part of the subprincipal symbol at
$q$, microlocal Sobolev regularity for solutions $u$ to $Pu = 0$ propagate to
$q$. Assuming microlocal Sobolev regularity of $u$ at $q$ just above this
threshold, $u$ is smooth at $q$. Thus, a normal form for $P$ should
depend on a subprincipal term, at least varying along $\Lambda$. As we
show here, the microlocal equivalence classes of $\RR^m(X, \Lambda)$
are essentially parametrized by subprincipal behavior at $\Lambda$.

For the remainder of the paper, we let $(y, z)$ be coordinates on
$\R^{n-1}_y \times \R_z$, with dual coordinates $(\eta, \zeta)$ for
the cotangent fibers. Given any $p_0 \in C^\infty(\R^{n-1})$,
$$z D_z + p_0(y)$$
is radial on $N^*\{z = 0\}$, as the Hamilton vector field is
$z \partial_z - \zeta \partial_\zeta$. All $P \in \RR^m(X, \Lambda)$
are microlocally equivalent to some such operator.

\begin{thm} \label{thm:normalform}
Given $P \in \RR^m(X, \Lambda)$ and $q \in \Lambda$, $P$ at
$q$ is microlocally equivalent to $zD_z + p_0(y)$ at $q_0 = (y = 0, z
= 0, \eta = 0, \zeta = 1)$ for some $p_0 \in C^\infty(\R^{n-1})$.
\end{thm}

The proof will proceed as follows. First, we will show that all
operators in $\RR^m(X, \Lambda)$ are microlocally equivalent at the
level of principal symbols. This is a two-step argument, following
those of Guillemin and Schaeffer (\cite[Section
3]{guilleminschaeffer}), and Hassell, Melrose, and Vasy (\cite[Section
3.2]{hmv1}). These in turn adopt Nelson's proof of the Sternberg
linearization theorem (\cite[Chapter 3]{nelson_flows}), and we use
this work directly at times. 

In the first step, we show that we can construct appropriate symbols at the level of formal power series expansions in degree of vanishing at the Lagrangian submanifold. By Borel's lemma, this is equivalent to achieving the normal form for our principal symbol up to an error term vanishing to infinite order on the Lagrangian submanifold. In the second step, we solve away this error term, again adapting the above sources.

After achieving our normal form at the principal symbol level, we
follow the methods of \cite{guilleminschaeffer} to deal with
lower-order terms in the full expansion. Our assumption that all
pseudodifferential operators have classical, one-step homogeneous
expansions is needed precisely here; at the principal symbol level,
only the principal symbol of the operator need be homogeneous. One
should be able to make similar statements without assuming that our
operators are classical. Here we again eliminate lower-order terms by a formal power
series argument, and then we can exploit the results in
\cite{guilleminschaeffer} directly to solve away the error which
vanishes to infinite order at the Lagrangian submanifold. Indeed, much
of this argument amounts to checking that \cite{guilleminschaeffer}
carries over to our setting.

\subsection{Further background}

Here we provide some further background which is needed/convenient for
the proof of Theorem~\ref{thm:normalform}

Given a smooth manifold $X$, let 
$$\kappa: T^*X \backslash o \rightarrow (T^*X \backslash o)/\R_{>0} =
S^*X$$
be the quotient map identifying the orbits of $\mu$ as defined by
\eqref{R+action}. $S^*X$ is called the cosphere bundle of $X$. As is
standard, $S^*X$ has a canonical contact structure. Given two smooth manifolds $X$ and $Y$, there is a natural correspondence between contact transformations
$$S^*X \rightarrow S^*Y$$
and homogeneous symplectomorphisms
$$T^*X \backslash o \rightarrow T^*Y \backslash o.$$
For a discussion of this, see \cite{guilleminschaeffer}, just after Lemma~4.2.

In \cite[Sections 1 and 5]{melroseasymp}, Melrose introduces the compactified cotangent bundle of a smooth manifold $X$, which we denote by $\ccot$. $\ccot$ is a disk bundle, whose interior is naturally identified with $T^*X$, and whose boundary is naturally identified with $S^*X$. Given coordinates $(y, z, \eta, \zeta)$ for $T^*\R^n$ as introduced immediately before the statement of Theorem~\ref{thm:normalform}, we can define new coordinates in $\zeta \neq 0$: we keep $y$ and $z$, and
\begin{align}
\theta & = \frac{\eta}{\zeta} \label{coordinates1} \\
\rho& = \frac{1}{\zeta} \label{coordinates2}
\end{align}

We can then extend this coordinate chart to the boundary of $\ccotr$ by setting $\rho = 0$ on the boundary. $\rho$ is then a local boundary defining function for $\ccotr$.

Please see \cite[Sections 1.3 and 2.1.1]{lag_rad2} for a further
description of the cosphere bundle and compactified cotangent
bundle. As is gone over in detail there, the Hamilton vector field
induces flows on both the cosphere bundle and the compactified
cotangent bundle. This is easy to see: as $H_{\sigma_m(P)}$ is
homogeneous of degree $m-1$, we may rescale it with an elliptic factor
to make it homogeneous of degree $0$. This extends to be a b-vector
field on $\ccot$. This gives a flow $\phi_t$, which depends on the
elliptic factor chosen, but different choices just give a
reparametrization. In \cite{lag_rad2}, Vasy and the author show that
$L = \partial \overline{\Lambda}$, the boundary of $\Lambda$ in $\ccot
\backslash o$ (identifying the boundary of $\ccot$ with $S^*X$, this
is the image of $\Lambda$ under the quotient map $\kappa$), is either
a submanifold of sinks or sources under this flow. Assuming, for
concreteness, that it is a manifold of sinks, one can then ask the
following question.
\begin{itemize}
	\item Is the map
$$x \mapsto \lim_{t \rightarrow \infty} \phi_t(x),$$
defined in a neighborhood of $\Lambda$, smooth across $L$?
\end{itemize}
Given the appearance of resonances in the isolated ray case, we have
another, related, question.
\begin{itemize}
  \item Do such resonances appear for the case of $\RR^m(X, \Lambda)$?
\end{itemize}
Theorem~\ref{thm:normalform} answers the first question in the affirmative, and the second in the negative.

\section{Principal symbol argument}

In this section, we prove that all operators in $\RR^m(X, \Lambda)$
are microlocally equivalent at the level of principal symbols. Given
$P \in \RR^m(X, \Lambda)$ and $q \in \Lambda$, we can multiply it by
an elliptic operator of order $1 - m$, and thus we can assume that $P
\in \RR^1(X, \Lambda)$. As is standard (see, for instance \cite[Theorem 21.2.8]{hormander3}),
we can choose an homogeneous local symplectomorphism to a neighborhood
in which the coordinates \eqref{coordinates1} - \eqref{coordinates2}
are valid, which maps $q$ to $q_0 = (y = 0, z = 0, \eta = 0, \rho =
1)$ and $\Lambda$ to a conic neighborhood of $q_0$ inside $\{z = 0,
\theta = 0\}$. We can then choose a Fourier integral operator
corresponding to this local symplectomorphism, which is elliptic
locally in this region, and reduce to the case where $P \in
\RR^1(\R^n, \{z = 0, \theta = 0, \zeta > 0\})$ and the point of
interest $q$ is $q_0$. We set
\begin{equation} \label{fixLambda}
\Lambda = \{z = 0, \theta = 0, \zeta > 0 \},
\end{equation}
$\overline{\Lambda}$ its closure in $\ccotr \backslash o$, and
\begin{equation*}
L = \{z = 0, \theta = 0, \rho = 0\} = \partial(\overline \Lambda) \subset \ccotr \backslash o.
\end{equation*}
We let
\begin{equation} \label{fixq0}
q_\infty = (z = 0, y = 0, \theta = 0, \rho = 0),
\end{equation}
i.e. the image of $q_0$ in $S^*\R^n$.

In what follows, we switch between considering the symplectic
structure on $T^*\R^n \backslash o$ and the contact structure on
$S^*\R^n$, depending on what is convenient. We thus define the
rescaled principal symbol of $P$ by $p = \rho \sigma_1(P)$, chosen to
be homogeneous of degree $0$. We identify it with a smooth function on
$S^*\R^n$, as we do generally for all functions on $T^*\R^n \backslash o$
which are homogeneous of degree $0$.

The following lemma shows that all such $P$ are microlocally
equivalent at the level of principal symbol, up to an error term vanishing to infinite order at the radial set $\Lambda$. Let $\II$ be the ideal of smooth functions on $S^*\R^n$ vanishing on $L$. That is,
$$\II = z C^\infty(S^*\R^n) + \sum_i \theta_i C^\infty(S^*\R^n).$$
As above, we identify this with all homogeneous degree $0$
smooth functions on $T^*\R^n \backslash o$ which vanish on $\Lambda$.

\begin{lem} \label{lem:normalwitherror}
There exists a local homogeneous symplectomorphism $\phi$ which fixes
$q_0$ and a homogeneous degree zero smooth function $e$ with $e(q_0) \neq 0$ so that $e \rho \phi^*(\rho^{-1} p) = z \mod \II^{\infty}$.
\end{lem}

\begin{proof}[Proof of Lemma~\ref{lem:normalwitherror}]
Given $a \in C^\infty(\ccotr \backslash o)$, we have
\begin{equation} \label{hamilton}
H_{\rho^{-1} a} = \partial_z a (\rho \partial_\rho + \theta \cdot \partial_\theta) - (\rho \partial_\rho a + \theta \cdot \partial_\theta a - a) \partial_z + \sum_i \partial_{\theta_i} a \partial_{y_i} - \partial_{y_i} a \partial_{\theta_i}.
\end{equation}
Requiring that the Hamilton vector field is radial on $\Lambda$ then amounts to $\partial_\theta p|_{\Lambda} = \partial_y p|_{\Lambda} = 0$, and our nondegeneracy condition lets us assume that $\partial_z p|_\Lambda \neq 0$. The radial condition then gives us that $p = \lambda(y) z \mod \II^2$, and the nondegeneracy condition gives us that $\lambda(q) \neq 0$. We can then absorb $\lambda$ into $e$, and assume that 
$$p = z \mod \II^2.$$

The symplectomorphism $\phi$ will be the time $1$ flow of a Hamilton
vector field corresponding to symbol $\rho^{-1} b$ (and thus the FIO
is $e^{i B}$ where $B$ is a quantization of $b$, with $b \in
C^\infty(S^*\R^n)$). As discussed in \cite[Section 2]{guilleminschaeffer},
there is a Lagrange bracket on $C^\infty(S^*\R^n)$ with a convenient
relation to the Poisson bracket. Given $a, b \in C^\infty(S^*\R^n)$, let
$\dpoiss{a,b} = \rho \{\rho^{-1} a, \rho^{-1} b\}$ (as above, we
identify smooth functions on $S^*\R^n$ with smooth functions on $T^*\R^n
\backslash o$ which are homogeneous of degree $0$ to make sense of
this formula). Note that by our explicit formula \eqref{hamilton},
$$\dpoiss{\cdot,\cdot}: \II^i \times \II^j \rightarrow \II^{i + j -
  1}.$$

If $\phi$ is the time $1$ flow of $H_{\rho^{-1} b}$, with $b \in \II^l$, $l > 0$, then 
\begin{equation}
\rho \phi^*(\rho^{-1} p) = p + \dpoiss{b, z} \mod \II^{l + 1}.
\end{equation}
To see this, note that we have the following formula, which converges in the sense of formal power series with respect to the grading $\II^l/\II^{l+1}$:
$$\rho (\exp{H_{\rho^{-1} b}}) \rho^{-1} = \Id + \mathrm{ad} \  b + \frac{1}{2}(\mathrm{ad}  \
b)^2 + \ldots$$
where $\mathrm{ad} \ b = \dpoiss{b, \cdot}.$ Thus all other terms are
higher order with respect to this grading.

Thus, to finish proving the lemma, we need simply to show that we can
choose $b$ to kill all terms of order $2$ and higher in the formal
power expansion with respect to the grading $\II^l/\II^{l+1}$. Having
shown that, we can then use Borel's lemma to promote $b$ to
$C^\infty(S^*\R^n)$ so that the time 1 flow kills off all but an error
term in $\II^{\infty}$.

Note that
$$\dpoiss{z, b}  = - b + \{ \rho^{-1} z, b\},$$
and by \eqref{hamilton},
\begin{align*}
\{\rho^{-1} z, \theta_i\} & = \theta_i \\
\{\rho^{-1} z, f(y)\} & = 0 \\
\{\rho^{-1} z, z\} & = z.
\end{align*}
We then have, for an arbitrary $f(y) z^a \theta^\alpha$, $a + |\alpha| = l$, $f$ smooth,
$$\dpoiss{z, f(y) z^a \theta^{\alpha}} = (a + |\alpha| - 1) f(y) z^a \theta^{\alpha}.$$
For $l> 1$, this coefficient is nonzero, so $b$ can be chosen as desired.

\end{proof}

We now remove the error term.

\begin{lem} \label{lem:principalnormal}
There exists a local homogeneous symplectomorphism $\phi$ which fixes
$q_0$ and $e \in T^*\R^n \backslash o$, homogeneous of degree $0$,  with $e(q_\infty) \neq 0$ so that $e \rho \phi^*(\rho^{-1} p) = z$.
\end{lem}

To prove this, it suffices to find a contact transformation for which
the pushforward of the contact vector field of $p$ is the contact
vector field of $z$. To do this, we make a small modification of
Nelson's proof of the Sternberg linearization theorem \cite[Chapter
2]{nelson_flows}. First, we ignore the contact structure, and simply prove a
statement about vector fields. The proof then carries over into the
contact setting, using a result of \cite{guilleminschaeffer}.

\begin{prop} \label{prop:linearization}
Let $X$ and $X_0$ be smooth vector fields on $\R^k$, such that $X_0:
\R^k \rightarrow \R^k$ (using the canonical trivialization of $T\R^k$) is linear and
\begin{equation} \label{infiniteordervanishing}
X = X_0 + o(\|z - L\|^\infty),
\end{equation}
where $L$ is the null space of $X_0$. Then there exists a local
diffeomorphism $\phi: (U, 0) \rightarrow (V, 0)$ (that is, sending $0$
to $0$ and a neighborhood $U$ of $0$ to a neighborhood $V$ of $0$) so that $\phi_* X = X_0$.
\end{prop}

\begin{rmk}
This statement is perhaps not the most natural one to make. A vector field $X$ on $\R^k$ which vanishes on a linear subspace $L$ induces a map $\II_L / \II_L^2$, i.e. there is a `linearization' $X_0: N^*L \rightarrow N^*L$, which can be identified with a vector field on $\R^k$. Assuming $X - X_0$ vanishes to infinite order on $L$, a local diffeomorphism should exist which pushes forward one vector field to the other. Proving this, however, would require a bit more work, and Proposition~\ref{prop:linearization} suffices for our purposes, since the contact vector field associated to $z$ is linear in our coordinates.
\end{rmk}

To prove this, we use this technical fact due to Nelson.

\begin{thm}[\cite{nelson_flows}, Chapter 3, Theorem 8] \label{thm:nelson}
Let $X$ be a $C^\infty$ vector field on $\R^k$ (thought, via the
canonical trivialization of $T\R^k$, as $X:\R^k \rightarrow \R^k$), with $X(0) = 0$, such
that each $d^jX$ satisfies a global Lipschitz condition (that is, for
each $j$, there is a $C_j$ so that for all $x, y \in \R^k$, $\|d^jX(x)
- d^jX(y)\| \leq C_j \|x - y\|$). Let $X_0 x = dX(0) x$, let $U(t)$ and $U_0(t)$ be the flows generated by $X$ and $X_0$, and define $X_1$ by $X = X_0 + X_1$. Suppose there is a linear subspace $N$, invariant under $X_0$, and a positive integer $l$ such that for all $m \geq 0$ and $j = 0, 1, 2, \ldots$ there is a $\delta > 0$ such that if $\|z - N\| \leq \delta$ then
$$\|d^j X_1(z)\| \leq \|z - N\|^m \|z\|^l.$$
Let $E$ be the linear subspace of all $x$ in $\R^k$ such that
$$\lim_{t \rightarrow \infty} \|U_0(t)x - N\| = 0.$$
Then for all $j = 0, 1, 2, \ldots$ and $x$ in $E$,
$$d^j(U(-t) U_0(t) x)$$
converges as $t \rightarrow \infty$ and the limit is continuous in $x$ for $x \in E$. Let
$$W_-(x) = \lim_{t \rightarrow \infty} U(-t) U_0(t) x$$
for $x \in E$. Then $W_-$ has a $C^\infty$ extension $G$ to $\R^k$ which is the identity to infinite order in a neighborhood of $0$ in $N$ and such that in a neighborhood of $0$ in $E$, $G^{-1}_* X = X_0$ to infinite order.
\end{thm}

\begin{proof}[Proof of Proposition~\ref{prop:linearization}]
As in the statement of Theorem~\ref{thm:nelson}, let $U_0(t)$ be the flow of $X$ and $X_0$, respectively.
Let $E^-_L$ and $E^+_L$ be the stable and unstable subspaces for $L$ with respect to $X_0$, that is,
$$E^\pm_L = \{x \in \R^k \ | \ \lim_{t \rightarrow \pm \infty} U_0(t) x \in L\}.$$
We have $\R^k = E^+_L \oplus E^-_L$.

Let $f \in C^\infty_c(\R^k)$ be a compactly supported cutoff function, equal to $1$ in a neighborhood of $0$. Then $-X_0 - f(X - X_0), -X_0,$ and $N = L$ satisfy the hypotheses of Theorem~\ref{thm:nelson} with $E = E^-_L$. Thus there exists a local diffeomorphism $G$ at $0$, which is an extension of $\lim_{t \rightarrow \infty} U(-t) U_0(t)$ from $E^-_L$ to all of $\R^k$, so that $(G^{-1})_* X = X_0$ to infinite order in a neighborhood of $0$ in $E^-_L$.

Thus we can assume that $X = X_0$ to infinite order in a neighborhood $U \subset E^-_L$ of $0$. Let $f$ be as above, but with its support small enough so that its intersection with $E^-_L$ is contained in $U$. Let $\tilde X = X_0 + f(X - X_0)$. Then $\tilde X - X_0$ vanishes to infinite order on $E^-_L$, so $\tilde X, X_0,$ and $N = E^-_L$ satisfy the hypotheses of Theorem~\ref{thm:nelson} with $E = \R^k$. Thus $G = \lim_{t \rightarrow \infty} U(-t) U_0(t)$ exists and is smooth on $\R^k$, and $G^{-1}_* X = X_0$.
\end{proof}

In order to extend this proof to the contact setting, we use the following theorem of Guillemin and Schaeffer:

\begin{thm}[\cite{guilleminschaeffer}, Section 3, Theorem 3] \label{thm:guilleminschaeffer}
Let $X$ be a contact manifold, $Y$ a closed submanifold of $X$ and $\phi_i$ a sequence of contact transformations such that the $k$ jets $j^k \phi_i$ converge to continuous limits on $Y$. Then there exists a neighborhood $U$ of $Y$ and a contact transformation $\phi: U \rightarrow X$ such that $j^k \phi|_Y = \lim j^k \phi_i |_Y$, for all $k$.
\end{thm}

\begin{proof}[Proof of Lemma~\ref{lem:principalnormal}]
By Lemma~\ref{lem:normalwitherror}, we can assume that $p - z$
vanishes to infinite order on $L = \overline{\Lambda}
\backslash \partial \Lambda$. As contact transformations on the
cosphere bundle correspond to homogeneous symplectomorphisms on $T^*X
\backslash o$, we exhibit a contact transformation, with $p$ and $z$
considered to be local functions on the cosphere bundle. Let $\chi \in
C^\infty_c(\R^{2n-1})$ be a compactly supported cutoff function,
identically $1$ in a neighborhood of $q_\infty = 0$. We can then replace $p$ by $z + \chi(p - z)$, as these functions agree locally around $0$, so the same contact transformation works for both. Let $X_p = H_{\rho^{-1} p}|_{\{\rho = 0\}}$ and $X_z = H_{\rho^{-1} z}|_{\{\rho = 0\}}$ denote the corresponding contact vector fields;  it suffices to find a contact transformation $G$ so that $G^{-1}_* X_p = X_z$.

The proof of Proposition~\ref{prop:linearization} achieves this, with two changes. First, instead of defining new vector fields with cutoffs, we simply control the support of $\chi$. For each application of Theorem~\ref{thm:nelson}, we can apply the theorem directly to $X = X_p$ and $X_0 = X_z$, so $U(t)$ and $U_0(t)$ are contact transformations for each $t$. For the first application of Theorem~\ref{thm:nelson}, Theorem~\ref{thm:guilleminschaeffer} then assures us that we can choose the diffeomorphism to be a contact transformation. For the second application of Theorem~\ref{thm:nelson}, the diffeomorphism constructed is automatically a contact transformation.
\end{proof}

\subsection{Lower-order terms argument}

In this section, we deal with all lower order terms in the homogeneous
expansion for the full symbol of $P$, completing the proof of
Theorem~\ref{thm:normalform}. As mentioned above, this argument is
essentially a verification that the methods of \cite[Section
4]{guilleminschaeffer} carry over to our setting. The argument will
again involve solving away lower order terms up to an error which
vanishes to infinite order on the Lagrangian submanifold of radial points $\Lambda$, and then getting rid of this error term. We begin this section with a technical lemma to be used at the end of the proof, which is our analogue of \cite[Section 4, Theorem 2]{guilleminschaeffer}.

\begin{lem} \label{lem:vanishingerror}
Let $V$ be a linear vector field on $\R^k$ vanishing on a subspace $L$, and let $c \in \R$. Then given $g \in C_c^\infty(\R^k)$ vanishing to infinite order at $L$, there exists a function $f \in C^\infty(\R^k)$ vanishing to infinite order at $L$ such that
$$Vf + c f = g$$
everywhere.
\end{lem}

To prove this, we use the following theorem of Guillemin and Schaeffer.

\begin{thm}[\cite{guilleminschaeffer}, Section 4, Theorem 4] \label{thm:gstechnical}
Let $U(t)$ be a group of linear transformations acting on $\R^k$. Let $N$ be a subspace of $\R^k$ invariant under $U(t)$ and let $E$ be the subspace of $\R^k$ consisting of all $x \in \R^k$ such that
$$\|U(t) x - N\| \rightarrow 0$$
as $t \rightarrow \infty$. Let $g$ be a compactly supported function on $\R^k$ which vanishes to infinite order along $N$. Set
$$f(x, s) = - \int^s_0 e^{c t} g(U(t) x) \ dt.$$
Then for all multiindices $\beta$, $\lim_{s \rightarrow \infty} D^\beta f(x, s)$ converges absolutely for all $x \in E$ and is a smooth function of $x$. Moreover this limit vanishes to infinite order on $N$.
\end{thm}

\begin{proof}[Proof of Lemma~\ref{lem:vanishingerror}]
As in the proof of Proposition~\ref{prop:linearization}, we let
$$E_L^\pm = \{x \in \R^k \ | \ \lim_{t \rightarrow \pm \infty} U_0(t) x \in L\}.$$
We have $\R^k = E_L^+ \oplus E_L^-$. Apply Theorem~\ref{thm:gstechnical} to $N = L$, $c$ replaced by $-c$, and $U(t)$ the time $t$ flow of $-V$, with $E = E_L^-$. For each multiindex $\beta$ and $x \in E_L^-$, let $f^\alpha(x) = \lim_{t \rightarrow \infty} D^\alpha f(x, s)$. We can construct $\tilde f$, defined on $\R^k$, so that $D^\alpha \tilde f(x) = f^\alpha(x)$ for all $x \in E_L^-$. For all $x \in E$, we have the following identity, for $x \in E$.
\begin{align*}
f(U(s) x) & = - \int^\infty_0 e^{-ct} g(U(t+s) x) \ dt \\
& = - e^{cs} \int^\infty_0 e^{-c(t + s)} g(U(t + s) x) \ dt \\
& = - e^{cs} \int^\infty_s e^{-c t} g(U(t) x) \ dt
\end{align*}
Differentiating both sides with respect to $s$ and setting $s$ equal to $0$, we obtain
$$- V \tilde f(x) =  c \tilde f(x) + g$$
on $E_L^-$. Repeating this argument with all partial derivatives of $\tilde f$, we see that $g + V \tilde f + c f$ vanishes to infinite order at $E_L^-$. We have thus reduced the proof to the case where $g$ vanishes on $E_L^-$. We can then apply Theorem~\ref{thm:gstechnical} with $N = L$, $U(t)$ the time $t$ flow of $V$, and $E = \R^k$ (this time, we can keep $c$ as is instead of replacing it by $-c$). By the same argument as above, $Vf + cf = g$ on $\R^k$.
\end{proof}

We now complete the proof of Theorem~\ref{thm:normalform}.

\begin{proof}[Proof of Theorem~\ref{thm:normalform}]

Lemma~\ref{lem:principalnormal} allows us to assume that $P - P' \in
\Psi^0(\R^n)$, where $\WF'(P')$ is contained in a  conic neighborhood of
$q_0$, and $P' = zD_z$ microlocally on a smaller  conic neighborhood of
$q_0$. Recall that we have fixed $q_0$ by \eqref{fixq0}, and that we
have fixed $\Lambda$ by \eqref{fixLambda}.

Let $p_0$ be the homogeneous representative of $\sigma_0(P - P')$
(here we use for the first time that $P$ has a one-step
expansion). For the remainder of the proof, we assume that $p_0$
vanishes on $\Lambda$, and we show that such a $P$ is microlocally
equivalent to $P'$, which implies the general statement and hence
completes the proof of the theorem.

To deal with lower-order terms, we successively conjugate $P$ by
pseudodifferential operators and multiply $P$ by elliptic factors so
that the resulting operator is $P'$ up to successively lower orders
within a fixed neighborhood $U$ of $q_0$. We choose these operators
acting on $P$ to be successively lower order perturbations of the
identity operator so that the series of correcting terms can be
asymptotically summed to yield $E, A \in \Psi^0(\R^n)$, elliptic at
$q_0$, so that 
\begin{equation} \label{lowerorderconj}
q_0 \notin \WF'(E A P A^{-1} - P').
\end{equation}
 Here $A^{-1}$ is a microlocal parametrix of $A$ at $q_0$. It is important that our
successive corrections are valid in a fixed neighborhood $U$ so that
asymptotic summation is possible. As above, all these operators are
classical, i.e., have 1-step expansions.

We proceed to construct $A_0 \in \Psi^0(\R^n)$, elliptic in $U$, and
$F_0 \in \Psi^{-1}(\R^n)$ so that $P^{(0)} = (\Id + F_0) A_0 P
A_0^{-1}$ is such that $P^{(0)} - P'$ is of order $-1$ in $U$
($A_0^{-1}$ is a parametrix of $A_0$ in $U$). We then
inductively assume, for $k \geq 0$, that $P^{(k)} - P'$ is of order
$-k-1$ in $U$ and construct $B_{k
  + 1} \in \Psi_{cl}^{-k-1}(\R^n)$ and $F_{k+1} \in
\Psi_{cl}^{-k-2}(\R^n)$ so that 
$$P^{(k+1)} = (\Id + F_{k+1})(\Id + iB_{k+1}) P^{(k)}(\Id + i
B_{k+1})^{-1}$$
is such that $P^{(k+1)} - P'$ is of order $-k - 2$ in $U$ ($(\Id + i
B_{k+1})^{-1}$ is a parametrix for $\Id + i B_{k+1}$ in $U$). Assuming we can do this, asymptotic
summation completes the proof. We can define
$$A = \biggl(\prod^\infty_{k = 1} (\Id + i B_k) \biggr) A_0,$$
where the infinite product is successively on the left, i.e.
$$\ldots (\Id + i B_3)(\Id + i B_{2}) (\Id + i B_1),$$
via asymptotic summation:
$$A = A_0 + i B_{1} A_0 + i B_{2} A_0 + i B_{3} A_0 - B_2 B_1 + i B_4
A_0 - B_3 B_1 \ldots.$$
Similarly, we can define
$$E = \biggl( \prod^\infty_{k = 1} (\Id + i B_k)(\Id + F_{k-1})
\biggr) A^{-1}$$
(again the infinite product is successively on the left), via
asymptotic summation:
\begin{align*}
E & = A^{-1} + F_0 A^{-1} + i B_{1} A^{-1} + F_{1} A^{-1} + i B_2
A^{-1} \\
& \quad + i B_3 A^{-1} + F_1 B_1 A^{-1} + F_2 A^{-1} \ldots.
\end{align*}

We can in fact choose $A_0$ and all $B_k$ and $F_k$ to be
quantizations of homogeneous symbols. Let $a_0$ be the homogeneous
symbol corresponding to $A_0$, with $a_0 = e^{i b_0}$. Let $b_k$ be
the homogeneous symbols corresponding to $B_k$ for $k > 0$. For all $k
\geq 0$, let $\rho f_k$ be the homogeneous symbols corresponding to
$F_k$. Lastly, at each step $k \geq 0$, let $p_k =
\sigma_{-k}(P^{(k-1)} - P')$ (well-defined within $U$) be chosen to be
homogeneous. We would then like
\begin{align*}
0 & = \sigma_0((\Id + F_0) A_0 P A_0^{-1} - P') \\
& = \sigma_0( (\Id + F_0) ([A_0, P] A_0^{-1} + P) - P') \\
& = \sigma_0([A_0, P]A_0^{-1} + F_0 P + P  - P') \\
& = \frac{1}{i} \{a_0, \rho^{-1} z\} a_0^{-1} + f_0 z + p_0 \\
& = - \{\rho^{-1} z, b_0\} + f_0 z + p_0.
\end{align*}
For $k > 0$, we would like (note that the symbol is, by
inductive assumption, well defined in $U$)
\begin{align*}
0 & = \sigma_{-k-1}((\Id + F_{k+1}) (\Id + i B_{k+1}) P^{(k)} (\Id + i
B_{k+1})^{-1} - P') \\
& = \sigma_{-k - 1}((Id + F_{k+1}) ( [\Id + i B_{k+1}, P^{(k)}] +
P^{(k)}) - P') \\
& = \sigma_{-k-1}([i B_{k +1}, P^{(k)}] + F_{k+1} P^{(k)} + P^{(k)} -
P') \\
& = -\{\rho^{-1} z, b_{k+1}\} f_{k+1} z + p_{k+1}
\end{align*}

In summary, the desired construction follows from being able to choose
(including $k = 0$) $b_k$ and $f_k$ so that
\begin{equation} \label{thingtosolve}
\{\rho^{-1} z, b_k\} - z f_k = p_k
\end{equation}
in $U$.

Following \cite[Section 4]{guilleminschaeffer}, we solve these
equations in  two steps. First, we solve up to an error term which
vanishes to infinite order on $\Lambda$, and then we solve away this
error.

By equation~\eqref{hamilton} and the homogeneity of each $b_k$,  
$$\{\rho^{-1} z, b_k\} = (\rho \partial_\rho + \theta \cdot \partial_\theta - z \partial_z) b_k = (\theta \cdot \partial_\theta - z \partial_z + k) b_k.$$
We can thus reduce \eqref{thingtosolve} to solving
\begin{equation} \label{newthingtosolve}
(\theta \cdot \partial_\theta - z \partial_z + k) \tilde b_k + z \tilde f_k = \tilde p_k
\end{equation}
where $\tilde b_k = \rho^{-k} b_k$, $\tilde f_k = \rho^{-k} f_k$, and $\tilde p_k = \rho^{-k} p_k$ are homogeneous of degree $0$, and hence can be identified with functions of $y, \theta$, and $z$. Note that, for all $a \in \Z_{\geq 0}$ and multiindices $\alpha$,
$$(\theta \cdot \partial_\theta - z \partial_z + k) (\theta^\alpha
z^a) = (|\alpha| - a + k) \theta^\alpha z^a.$$

Thus, as formal power series in $\theta$ and $z$, we can solve
\eqref{newthingtosolve}, as 
\begin{itemize}
  \item all terms in $\tilde p_k$ with positive powers of $z$ can be
    absorbed by $f_k$,
    \item for $k = 0$, we use the assumption that $p_0$ vanishes at
      $\theta = 0, z = 0$, and
      \item for $k > 0$, the $k$ term allows us to take care of
        nonvanishing terms.
\end{itemize}
Hence we can solve
\eqref{newthingtosolve} up to a term which vanishes to infinite order
at $\theta = 0, z = 0$. Note that this can be done globally (at least,
where the coordinates are defined), so we have this in any
neighborhood $U$ in which the coordinates are defined, for all $k$.

It remains to show that there exists a neighborhood $U$ of $q_0$ so that, for all $k$ and $g(y, z, \theta)$ vanishing to infinite order at $\{\theta = 0, z = 0\}$, there exists $f(y, z, \theta)$ so that
$$(\theta \partial_\theta - z \partial_z + k) f = g$$
in $U$. This is precisely the setting of
Lemma~\ref{lem:vanishingerror}, as we can choose $\chi$ to be a
compactly supported cutoff function which is identically $1$ in some
neighborhood $U$ of the origin with support in which the coordinates
are defined, and replace $g$ by $\chi g$. This completes the
proof.

\end{proof}

\bibliographystyle{plain}

\bibliography{radial_points}

\end{document}